\documentclass[11pt, oneside]{amsart}
\usepackage{amssymb, amsfonts, amsthm,url}

\newtheorem{lemma}{Lemma}[section]
\newtheorem{proposition}[lemma]{Proposition}
\newtheorem{theorem}[lemma]{Theorem}
\newtheorem{defn}[lemma]{Definition}

\newtheorem{conjecture}[lemma]{Conjecture}
\newtheorem{corollary}[lemma]{Corollary}

\newtheorem{pbm}[lemma]{Problem}

\newtheorem*{special theorem}{My Specially-Named Theorem}

\begin{document}

% this starts the page-numbering at the bottom-center with roman numerals,
% like you're supposed to for the beginning pages

\pagestyle{plain}

% DOCUMENT INFO

\title{ A note on $d$-maximal $p$-groups I }
\author{Messab Aiech}
\address{Department of Mathematics\\
	ENSET-Skikda\\
	Skikda, Algeria.}
\email{aiech21messab@gmail.com}
\author{Hanifa Zekraoui}
\address{Department of Mathematics\\
	University Larbi Ben M’hidi \\ 
Oum El Bouaghi, Algeria.}
\email{Hanifazekraoui@yahoo.fr}
\author{Yassine Guerboussa}
\address{Department of Mathematics\\
	University Kasdi Merbah-Ouargla\\
	Ouargla, Algeria.}
\email{yassine\_guer@@hotmail.fr}

% TITLE PAGE
\maketitle

\begin{abstract}
A finite $p$-group $G$ is said to be $d$-maximal if $d(H)<d(G)$ for every  subgroup $H<G$, where $d(G)$ denotes the minimal number of generators of $G$. A similar definition can be formulated when  $G$  is acted on by some group $A$.  We generalize results of  B. Kahn and T. Laffey  to the latter case, and  give them in particular alternative short proofs.  We answer moreover a question of Y. Berkovich about the minimal non-metacyclic $p$-groups.

\end{abstract}\vspace{1cm}

%\keywords{Keywords: $p$-groups; $d$-maximal; powerful $p$-groups.  }

%\keywords{MSC 2010: 20D15  }

\section{Introduction}	
Throughout, $p$ denotes a prime number,  $G$ a   $p$-group, and $A$ a  group acting on $G$ (unless otherwise stated, all the groups considered here are finite).   

A subgroup of $G$ which is $A$-invariant will also be termed  an $A$-subgroup. The minimal number of generators of $G$ will be denoted by $d(G)$, so $d(G)=\dim_{\mathbb{F}_p} G/\Phi(G)$, where $\Phi(G)=G^{p} \gamma_{2}(G) $ is the Frattini subgroup of $G$.    We shall write $\Omega_n(G)$ for the subgroup $\langle x\in G \mid  x^{p^{n}}=1\rangle$, and $G^{p^{n}}$ for the subgroup  $\langle x^{p^{n}}  \mid  x\in G \rangle$.  We write  $P_n(G)$ for the $n$-th term of the lower $p$-central series of $G$; recall that $P_1(G)=G$, and $P_{n+1}(G)=[P_n(G),G] P_n(G)^{p}$, for $n\geq 1$.  Any other unexplained notation is standard. 

\begin{defn}
We say that  $G$  is \emph{$d$-maximal} for $A$-subgroups if   $d(H)< d(G)$ for every proper  $A$-subgroup $H$ of $G$. 
\end{defn}
 If the action of $A$  on $G$ is trivial, we cover the notion of $d$-maximal $p$-groups introduced by Kahn in  \cite[Appendice]{Kahn SW}.  These $p$-groups are interesting at least for the following reason: 

($\mathrm{F}$) \textit{There always exists  an $A$-subgroup $K\leq G$ such that $K$ is $d$-maximal for $A$-subgroups, and  $d(K)\geq d(G)$}. \footnote{This fact can be strengthened when $A$ is  a  $p$-group (see Lemma \ref{Second key}). }

Indeed, consider the  set of the $A$-subgroups $S\leq G$ such that $d(S)\geq d(G)$, and let $K$ be  any minimal element in this set  (with respect to the order defined by inclusion).  
If $H<K$ is $A$-invariant,  then   $d(H)<d(G)\leq d(K)$; the result follows.

In \cite{Kahn SW}, the interest  in $d$-maximal $p$-groups  arose from the study of the
Stiefel-Whitney classes $w_{i}(r_G) $ of the  real regular representation $r_G$ ($G$ is a $2$-group). More precisely,  Kahn  introduced  the following invariant:
$$\nu(G)=\min \{n>0\mid w_{2^{n-1}}(r_G)\neq 0\},$$
and conjectured the following.
\begin{conjecture}\label{Kahn conjecture} $\nu(G)\geq d(G)$.
\end{conjecture}
Among other things, he showed that this invariant
is increasing on subgroups, that is $\nu(H)\leq \nu(G)$ whenever $H\leq G$, so the conjecture reduces, by virtue of ($\mathrm{F}$),  to the  case where $G$ is  $d$-maximal. This conjecture is still widely open; we refer the reader to  Minh \cite{Minh,Minh1997} for more details. 
\medskip

For $p>2$, every   $d$-maximal $p$-group has nilpotency class $\leq 2$ (see \cite[Th\'{e}or\`{e}me A (a)]{Kahn SW}).\footnote{This result is also implicit in  Laffey  \cite{Laffey}.}  We start by  generalizing the latter to $p$-groups with operators.
\begin{theorem}\label{Main 1}
	Let $G$ be a  $p$-group, $p$ odd, and $A$ be   a $p$-group acting on $G$.   If $G$ is  $d$-maximal for $A$-subgroups, then $G$ has nilpotency class  $\leq 2$.  
\end{theorem}

As a consequence, we have the following generalization of the main result in \cite{Laffey} (the latter  follows by taking $A=\mathrm{Inn}(G)$).

\begin{corollary}\label{Coro Laffey}
	Let $G$ be a  $p$-group, $p$ odd, and $A$ be  a $p$-group acting on $G$.  Then there exists an $A$-subgroup $K\leq G$ such that $d(K)= d(G)$, $K$ has   class $\leq 2$, and $\gamma_2(K)=\Phi(K)$.  
\end{corollary}
Note that such a $K$ is $p$-abelian, i.e., the map $x\mapsto x^{p}$ is an endomorphism of $K$; this follows at once from the identity
$(xy)^{p}=x^{p}y^{p}[y,x]^{\binom{p}{2}}$ (as $K$ has class $\leq 2$),  and the fact that $\gamma_{2}(K)$ has exponent dividing $p$ (see Lemma \ref{Basic facts}(4)). Hence,    $\Omega_1(K)$ has exponent dividing $p$, and  $\vert K:K^{p}\vert = \vert \Omega_1(K)\vert $; in particular, $p^{d(K)}\leq  \vert \Omega_1(K)\vert$.   The latter implies that for every $A$-subgroup $H\leq G$, we have $d(H)\leq k$,  where $p^{k}$ is the  maximal order of an $A$-subgroup of $G$ which has exponent $p$ and  class $\leq 2$.
\medskip

The situation for $p=2$ turned out to be  much more intricate.  For instance, Minh (see \cite{Minh}) has constructed $d$-maximal $2$-groups of class $3$ (and order $2^{8}$);  the latter involves cohomology and subtle calculations with factor sets.  Nowadays, one can find several of such examples using  the Small Groups library of GAP (see \cite{GAP}); this issue will be discussed in  more detail in Section \ref{Last section}.  

General  results  on the structure of $d$-maximal $2$-groups were obtained by Kahn in \cite{Kahn APE}.  The second aim  of this note  is to give simpler and more transparent proofs of these results, and to generalize them at the same time  to $d$-maximal $2$-groups with operators.  First we shall adapt the proof of Theorem \ref{Main 1}  to prove the following key result.
\begin{theorem}\label{Kahn's modified 1}
	Let $G$ be a  $2$-group, and $A$ be  a $2$-group acting on $G$.   If $G$ is  $d$-maximal for $A$-subgroups, then $\gamma_3(G)=\gamma_2(G)^{2}$.  
\end{theorem}
Following \cite{Kahn APE}, we say that $N\leq G$ is \textit{almost powerfully embedded} in $G$  if	it satisfies: (i) $[N,G]\leq N^{2}$, and (ii) $[N,N]\leq (N^{2})^{2}$. One easily sees  that (ii) is equivalent to the fact that  $N$ is powerful, that is  $[N,N]\leq N^{4}$.  For the theory of powerful $p$-groups, we refer the reader to   \cite{MannLubotzky} or \cite[Chapter 2]{DDMS}.     

To keep the exposition as clear as possible, let us quote the following result from \cite{Kahn APE}.  

\begin{theorem}(\cite[Theorem 1(1)]{Kahn APE})\label{Kahn's almost powerfuly embed}
	Let $G$ be a  $2$-group, and $N\leq G^2$.  If $[N,G]\leq N^{2}$, then $[N,G^{2}]\leq (N^{2})^{2}$. In particular, $N$ is almost powerfully embedded in $G$.  
\end{theorem}
The proof of Theorem \ref{Kahn's almost powerfuly embed} given in \cite{Kahn APE}   is unnecessarily complicated; here is a  more direct one. 	   Without loss of generality, we may assume $[N,G^{2},G]=1$ and $[N,G^{2} ]^{2}=1$ (cf.  the introduction of Section \ref{By-product of being powerful}).  It follows in particular that $[N,N,G]=1$ and $[N,N]^{2}=1$.  Now, for $a,b\in N$, we have $[a^{2},b]=[a,b]^{2}=1$; hence $[N^{2},N]=1$, and so $[N,G,N]=1$.  Next, for $g\in G$, we have
$[a^{2},g]=[a,g]^{2}$, which lies in  $ (N^{2})^{2}$; thus $[N^{2},G]\leq (N^{2})^{2}$.  Finally, we have $[a,g^{2}]=[a,g]^{2}[a,g,g] $, but $[a,g]^{2}\in (N^{2})^{2}$ and $[a,g,g]\in  [N^{2},G]\leq (N^{2})^{2}$; thus $[N,G^{2}]\leq (N^{2})^{2}$, which  completes the proof.
\medskip

The last two theorems imply at once that if $G$ is $d$-maximal for $A$-subgroups ($A$ is a $2$-group), then $\gamma_2(G)$ is
almost powerfully embedded in $G$.  More generally, we have:

\begin{corollary}\label{Cor main 2}
  If $G$ is  $d$-maximal for $A$-subgroups, where $A$ is a $2$-group,  then  $\gamma_n(G)$ is almost powerfully embedded in $G$, and  $\gamma_n(G)=P_{n}(G)=G^{2^{n-1}}$, for all $n\geq 2$.	
\end{corollary}

Next, we have the following criterion for $d$-maximality.
\begin{corollary}\label{Equivalence for d-maximality}
	Let $G$ be a  $2$-group, and $A$  a $2$-group acting on $G$.  Then the following assertions are equivalent:
	\begin{itemize}
		\item[(i)] $G$ is $d$-maximal for $A$-subgroups.
		\item[(ii)] $G/\gamma_3(G)$ is $d$-maximal for $A$-subgroups.
		\item[(iii)]  $d(H)<d(G)$, for every  $A$-subgroup $H<G$ containing $\gamma_2(G)$.
	\end{itemize}  
\end{corollary}	
The above shows in particular that  $G$ is $d$-maximal if, and only if, it is $d$-maximal for normal subgroups (since every subgroup containing $\gamma_{2}(G)$ is normal).   Note also that the last two corollaries imply   \cite[Theorem 2]{Kahn APE}.
\medskip

As an application,  we shall answer a question of Berkovich (see \cite[Problem 1850 (i)]{Berk3}) about  the structure of the $2$-groups $G$ for which $G/G^{4}$ is minimal non-metacyclic  (a  group $X$ is termed    \textit{minimal  non-metacyclic} if it is not  metacyclic, but all its proper subgroups  are).  
\begin{theorem}\label{Nonmetacyclic}
	Let $G$ be a  $2$-group, and $A$  a $2$-group acting on $G$.
	\begin{itemize}
		\item[(i)] If $G/G^{4}$ is $d$-maximal for $A$-subgroups, then so is $G$.  
		\item[(ii)] If $G/G^{4}$ is  minimal non-metacyclic, then so is $G$.
	\end{itemize}  
\end{theorem}
The minimal non-metacyclic $2$-groups have been classified by  Blackburn  (see \cite[Theorem 3.2]{Black}; see also Janko \cite[Theorem 7.1]{Janko} for an alternative approach). Accordingly,  every such a $2$-group $G$ satisfies $d(G)=3$, so $G$ is $d$-maximal.  Conversely, let $G$ be a $d$-maximal $2$-group with $d(G)=3$; if $H<G$, then all the subgroups of $H$ can be generated by two elements, so $H$ is metacyclic (cf. e.g., \cite[Theorem 44.5]{Berk1}), and subsequently $G$ is minimal non-metacyclic.  

\begin{theorem}\label{Class of 3-gen d-max 2-grp}[Blackburn]
If $G$ is a $d$-maximal $2$-group with $d(G)=3$, then $G$ is  one of the following:
\begin{itemize}
	\item[(a)] the elementary abelian group $C_2\times C_2\times C_2$;
	\item[(b)] the direct product $C_2\times Q_8$;
	
	\item[(c)] the central product $C_4\ast Q_8=C_4\ast D_8$ (of order $2^{4}$);
	\item[(d)] The group of order $2^{5}$ defined by
	$$\langle  a,b,c \mid a^4= b^4=[a,b]=1,  c^2= a^2, a^c=ab^2,  b^c= ba^2\rangle .$$
\end{itemize}
\end{theorem}
Since \cite[Theorem 3.2]{Black} and \cite[Theorem 44.5]{Berk1} rely on several other results in the literature, it is desirable to have a more direct proof of the last theorem.  Such a proof will be given in Section \ref{Futher Results}; we shall first prove that these $2$-groups have class $\leq 2$ (Proposition \ref{dMax 2Grp on 3-gen have class 2}), so they are of order $\leq 2^5$; next we can easily deal  with each possible order (following  \cite[Theorem 7.1]{Janko}, for example).  This provides  in particular another  proof of Blackburn's \cite[Theorem 3.2]{Black}. Note also that one can  establish Proposition \ref{dMax 2Grp on 3-gen have class 2} differently, by using  cohomology as in \cite{Minh}.

In Section \ref{Futher Results}, we shall also determine the $d$-maximal $p$-groups $G$ with  $d(G)=3$, and $p>2$ (see Proposition \ref{dMax p-groups, 3 gen and p odd}).  More results concerning the case $p>2$,  will appear in a subsequent paper.  Theorems \ref{Main 1} and \ref{Kahn's modified 1},  and Corollary \ref{Coro Laffey} will be proved in Section \ref{Basic results}, and the  remaining ones  in Section \ref{By-product of being powerful}.    In the last section, we investigate  the small $d$-maximal $2$-groups (of order $\leq 2^9$)  by using GAP \cite{GAP}; we conclude with some open problems.

\section{Proof of Theorems \ref{Main 1} and \ref{Kahn's modified 1} }\label{Basic results}
First, we need some basic facts; these  are more or less  well-known (see \cite[Appendice]{Kahn SW}), although we give them proofs for the reader convenience.
\begin{lemma}\label{Basic facts}
	Let $G$ be a $p$-group acted on by a group $A$, and assume that $G$ is $d$-maximal for $A$-subgroups. 
	\begin{enumerate}
		\item If $G$ is abelian, then it is  elementary abelian.
		\item If $N$ is a normal $A$-subgroup of $G$ such that $N\leq \Phi(G)$, then $G/N$ is a $d$-maximal for $A$-subgroups.
		\item $\Phi(G)=[G,G]$.
		\item The factors $\gamma_n/\gamma_{n+1} $ of the lower central series $(\gamma_n)_{n\geq 1}$ of $G$ are elementary abelian.
		\item If $M$ is an $A$-subgroup of $G$ of index $p$, then $d(M)=r-1$, where $r=d(G)$.
		
	\end{enumerate}
\end{lemma}
\begin{proof}
	For (1), observe that $G_1=\{x\in G\mid x^{p}=1\}$ is an $A$-subgroup, and $d(G)=d(G_1)$; hence  $G=G_1$.  For (2), observe  that $d(H/N)\leq d(H)$ for every $A$-subgroup $H\leq G$ containing $N$, and that $d(G/N)=d(G)$ since $\Phi(G/N)=\Phi(G)/N$.  As  $[G,G]$ is $A$-invariant, (3) follows at once from (2) and (1).   Next, note that for every integer $n\geq 1$, we have  a canonical epimorphism $x \gamma_2\otimes y\gamma_n \mapsto [x,y]\gamma_{n+1}$, from $\gamma_1/\gamma_2 \otimes \gamma_n/ \gamma_{n+1}$ to  $\gamma_{n+1}/ \gamma_{n+2} $. Now, (4)  follows at once  from (3) and  induction on $n$.  For (5), observe  that  $\vert\Phi(G):\Phi(M)\vert <\vert G:M\vert=p$, so $\Phi(M)=\Phi(G)$; the result follows.
\end{proof}
\begin{proof}[Proof of Theorem \ref{Main 1}]

Assume  $G$ is a counterexample of  minimal order; in particular $\gamma_3(G)\neq 1$.  By viewing $\gamma_3(G)$   as a normal subgroup of  $\Gamma=G \rtimes A$, we obtain $[\gamma_3(G),\Gamma]<\gamma_3(G)$;  choose therefore a maximal subgroup $N\leq \gamma_3(G)$ which contains $[\gamma_3(G),\Gamma]$. Hence, $N$ is a normal $A$-subgroup of $G$,  $G/N$ is $d$-maximal by  Lemma \ref{Basic facts}(2),  and plainly $\gamma_3(G/N)\neq 1$.  The minimality of $G$ implies $N=1$,   so $\vert \gamma_3(G)\vert =p$.  

Let $Z=\gamma_2(G)\cap Z(G)$,  and  consider the map $g\mapsto g^{*}$, from $G$ to $\mathrm{Hom}(\gamma_2(G)/Z,\gamma_3(G))$, defined by $g^{*}(tZ)=[g,t]$ for all $t\in \gamma_2(G)$. The latter  is readily seen to be a group homomorphism  whose kernel is equal to  $C=C_G(\gamma_2(G))$.  Now, $G/C$ embeds in $\mathrm{Hom}(\gamma_2(G)/Z,\gamma_3(G))$, and the latter is isomorphic to $ \gamma_2(G)/Z$ (since $ \gamma_2(G)/Z$ is elementary abelian); thus $\vert G:C\vert \leq \vert \gamma_2(G):Z\vert$;  equivalently $p^{d(G)}\leq \vert C:Z\vert$.  We claim that $\Phi(C)\leq Z$. Once this is proved, it follows that $d(G)\leq d(C)$, which is absurd   since $C$ is  $A$-invariant and $C<G$ (as  $G$ has class $3$).

For the last claim,  observe that $[C,G,C]=[G,C,C]=1$, so by the three subgroups lemma, $[C,C,G]=1$,  that is  $[C,C]\leq Z$.  Hence, we have only to show that $C^{p}\leq Z$. Let $c\in C$ and $g\in G$.  As $G$ has class $3$, one easily sees  by induction on $n$  that  
$$[c,g^{n}]=[c,g]^{n}[c,g,g]^{\binom{n}{2}}.$$
For $n=p$,  we have   $[c,g^{p}]=1$ as $G^{p}\leq \gamma_2(G)$, and $[c,g,g]^{\binom{p}{2}}=1$ as $\vert \gamma_3(G)\vert =p$;  hence $[c,g]^{p}=1$.  But  $[c^{p},g]=[c,g]^{p}$ because $c$ commutes with $\gamma_2(G)$.  Thus $C^{p}\leq Z$, as desired.

\end{proof}

Next, Corollary \ref{Coro Laffey} 
is immediate  by Theorem \ref{Main 1},  Lemma \ref{Basic facts}(3), and the following.
\begin{lemma}\label{Second key} Let $G$ be a  $p$-group,  $A$ be a $p$-group acting on $G$, and  
	$$r=\max\{d(H)\mid H \mbox{ is an $A$-subgroup of $G$}\}.$$
	Then, for every  integer $0\leq i <r$,   there exists an $A$-subgroup $K_i\leq G$ such that  $d(K_i)= r-i$, and $K_i$  is $d$-maximal for $A$-subgroups.
\end{lemma}
\begin{proof}
	Let $H_0$  be an $A$-subgroup of $G$ such that $d(H_0)=r$. By ($\mathrm{F}$),  there exists a $d$-maximal $A$-subgroup $K_0\leq H_0$ such that    $d(K_0)\geq d(H_0)$. Hence, $d(K_0)=r$ as $H_0$ has the maximal possible rank.   Since $A$ is a $p$-group, we have $[K_0,A]<K_0$; consider therefore a maximal subgroup  $H_1$ of $K_0$ that contains  $[K_0,A]$, so $H_1$ is $A$-invariant, and $d(H_1)=r-1$  by Lemma \ref{Basic facts}(5).  Again,  by ($\mathrm{F}$),  there exists a $d$-maximal $A$-subgroup $K_1\leq H_1$ such that $d(K_1)\geq d(H_1)$; but $d(K_1)<r$ since $K_1<K_0$, hence  $d(K_1)=r-1$.  Applying the previous argument to $K_1$, and so forth, yields the desired result.  
\end{proof}

\begin{proof}[Proof of Theorem \ref{Kahn's modified 1}]

We have $\gamma_2(G)^{2}\leq  \gamma_3(G)$ by Lemma \ref{Basic facts}(4).  If the quotient   $\gamma_3(G)/ \gamma_2(G)^{2}$ is  non-trivial, then  $[\gamma_3(G),\Gamma]\gamma_2(G)^{2} <\gamma_3(G)$, where $\Gamma=G\rtimes A$ acts on $\gamma_3(G)/ \gamma_2(G)^{2}$ in the obvious way.  Choose a maximal subgroup $N$ of $\gamma_3(G)$   which contains $[\gamma_3(G),\Gamma]\gamma_2(G)^{2}$.  So, $N$ is  $\Gamma$-invariant,   $G/N$ is   $d$-maximal for $A$-subgroups (by Lemma \ref{Basic facts}(2)) and $G/N$ does not satisfy our claim. Thus we may  assume  $N=1$,  so $\gamma_2(G)^{2}=1$ and $\vert\gamma_3(G)\vert=2$.

Let  $Z=\gamma_2(G)\cap Z(G)$, $C=C_G(\gamma_2(G))$, and   consider the map $g\mapsto g^{*}$ from $G$ to $\mathrm{Hom}(\gamma_2(G)/Z,\gamma_3(G))$ as in the proof of Theorem \ref{Main 1}. It follows likewise that $2^{d(G)}=\vert G:\gamma_2(G)\vert \leq \vert C:Z\vert$.   Obviously, $C^{2}\leq \gamma_2(G)$; and for $c\in C$ and $g\in G$, we have  $[c^{2},g]=[c,g]^{2}=1$, so $C^{2}\leq Z$.  Now, $C$ is a proper  $A$-subgroup of $G$ which satisfies $d(G)\leq d(C)$, a contradiction.  Thus $\gamma_3(G)/ \gamma_2(G)^{2}=1$; the result follows. 
\end{proof}

\section{The result of being powerful}\label{By-product of being powerful}

Suppose $X$ and $Y$ are two normal subgroups of  $G$; it is well-known that  $$X\leq YX^{p}[X,G] \implies X\leq Y.$$ (Otherwise, one can find a maximal $G$-invariant subgroup $W\leq X$ that contains $X\cap Y$, but such a $W$ also contains $X^{p}$ and $[X,G]$, a contradiction.)  The point here is that in order to establish $X\leq Y$, one can suppose that $X^{p}=[X,G]=1$ (in other words, we  prove our claim in $G/X^{p}[X,G]$, and then deduce it for $G$ by the above implication).  This trick has been used in proving Theorem \ref{Kahn's almost powerfuly embed}; we shall use it to prove the next result as well.

\begin{lemma}\label{Generating almost p.e. subgroups}[cf. \cite[Corollary 1 ]{Kahn APE}]
	Let $G$ be a $2$-group, and $N,M\leq G$. 
	\begin{itemize}
		\item[(i)]  If $N$ is almost powerfully embedded in $G$, then so are $N^{2}$, $[N,G]$.
		\item[(ii)]  If $N,M\leq G^{2}$, and both of them are almost powerfully embedded in $G$, then so are $NM$ and $[N,M]$.
	\end{itemize}

\end{lemma}

\begin{proof}
	(i) Let $X=[N,G] $ or $N^2$.  We have $X\leq G^{2}$, so  by Theorem \ref{Kahn's almost powerfuly embed}, we have only to show that $[X,G]\leq X^{2}$.
	
	For $X=[N,G]$, we may assume  $[N,G,G,G]=1$.   It follows that $[N^{2},G]$ lies in $[N,G]^{2}$. Indeed, for  $a\in N$ and $g\in G$, we have
	$$[a^{2},g]=[a,g][a,g,a][a,g] =[a,g]^{2}[a,g,a].$$
	Thus, $[N^{2},G] \leq [N,G]^{2}[N^{2},N]$. Assuming $g\in N$, it follows from  the above identity that $[N^{2},N]\leq \gamma_2(N)^{2}\gamma_3(N)$.  Next, observe that $\gamma_3(N)\leq [N^{4},N]$; so  replacing $a$ by $a^{2}$  yields $[a^{4},g]=[a^{2},g]^{2}[a^{2},g,a^{2}]$, but $[a^{2},g,a^{2}]=[a^{2},g,a]^{2}$.   Now, we have $\gamma_3(N)\leq \gamma_2(N)^{2}\leq [N,G]^{2}$.  Thus, $[N,G,G]\leq [N^{2},G]\leq[N,G]^{2}$, as claimed.  
	
	For $X=N^{2}$, we may assume   $[N^{2},G,G]=1$, and in particular $[N,G,G,G]=1$;   therefore $[N^{2},G]\leq [N,G]^{2}\leq (N^{2})^{2}$.
	
	(ii) Similarly, by Theorem \ref{Kahn's almost powerfuly embed},  we need only to show that $[NM,G]\leq (NM)^{2}$, and $[N,M,G]\leq [N,M]^{2}$.  The first case is trivial. For the second, we may assume $[N,M,G,G]=1$.  We claim that $[N^{2},M]\leq [N,M]^{2}$. Once this is proved, it follows by symmetry that  $[M^{2},N]\leq [N,M]^{2}$; so both of $[N,G,M]$ and $[M,G,N]$  lie in $[N,M]^{2}$,  and by the three-subgroups lemma, the same is true for $[N,M,G]$.  To prove that claim, let $a\in N$,  $b\in M$, and $g\in G$. We have
	$[a^{2},b]=[a,b]^{2}[a,b,a]$, and $[a,b,g^{2}]=[a,b,g]^{2}$; so $$[N^{2},M]\leq [N,M]^{2} [N,M,N]\leq [N,M]^{2} [N,M,G^{2}]\leq [N,M]^{2}.$$
\end{proof}
Note that (i) was proved in \cite[Corollary 1 ]{Kahn APE} under the assumption $N\leq G^{2}$. For (ii), the assumption  $N,M\leq G^{2}$ is necessary for $NM$ to satisfy the claim (consider for instance  $G=Q_8$, and $N,M$ two  distinct maximal subgroups of $Q_8$); we don't know if that condition is also  necessary for $[N,M]$.

\begin{proof}[Proof of Corollary \ref{Cor main 2}]
We proceed by induction on $n$.  For $n=2$, the result is immediate by Theorems \ref{Kahn's modified 1} and \ref{Kahn's almost powerfuly embed}, and Lemma \ref{Basic facts}(3).  Next, assume  that  $P_{n}(G)=\gamma_n(G)$ is almost powerfully embedded in $G$, for some  $n\geq 2$.  It follows that $P_{n+1}(G)=P_n(G)^2$, and we have $P_n(G)^2=\gamma_n(G)^{2}\leq \gamma_{n+1}(G)$ by lemma \ref{Basic facts}(4);  so $P_{n+1}(G)\leq \gamma_{n+1}(G)$.  Conversely, 
$$ \gamma_{n+1}(G)=[P_{n}(G),G]\leq P_{n}(G)^{2}=P_{n+1}(G),$$ 
so, $P_{n+1}(G)= \gamma_{n+1}(G)$.  Moreover, $\gamma_{n+1}(G)$ is almost powerfully embedded in $G$ by Lemma \ref{Generating almost p.e. subgroups}(i).  This shows that $P_{n}(G)=\gamma_n(G)$ is almost powerfully embedded in $G$, for all  $n\geq 2$.  

Next, assume that $G^{2^{n-1}}=P_n(G)$ holds for some $n\geq 2$. Clearly,  we have $G^{2^{n}}\leq (G^{2^{n-1}})^2 \leq P_{n+1}(G)$.   Conversely, since
$2n\geq n+2$, we have
$$[P_n(G), P_n(G)] \leq P_{2n}(G) \leq P_{n+2}(G)= P_{n+1}(G)^2.$$
Thus, $P_{n}(G)/P_{n+1}(G)^2$ is abelian.  Now, let $x\in P_{n}(G)=G^{2^{n-1}}$, so there exist $x_1,\ldots,x_s\in G$ such that $x=x_{1}^{2^{n-1}}\ldots x_{s}^{2^{n-1}}$.  It follows that
$$x^{2}=x_{1}^{2^{n}}\ldots x_{s}^{2^{n}} \mod  P_{n+1}(G)^2;$$
so $P_{n}(G)^2=G^{2^{n}} P_{n+1}(G)^2$, that is $P_{n+1}(G)=G^{2^{n}} \Phi(P_{n+1}(G))$. Thus  $P_{n+1}(G)=G^{2^{n}}$.   This shows that $G^{2^{n-1}}=P_n(G)$ for all $n\geq 2$, and
completes the proof. 
\end{proof}

In \cite[Theorem 2]{Kahn APE}, Kahn  claimed that the equality $P_n(G)=G^{2^{n-1}}$ follows from \cite[Theorem 4.1.3]{MannLubotzky} ; the latter seems to yield only $P_n(G)=(G^{2})^{2^{n-2}}$.

\medskip

Next, we shall prove Corollary \ref{Equivalence for d-maximality}.  We need the following result (which we shall use  as a substitute for  \cite[Proposition 1]{Kahn APE}).
\begin{proposition}\label{Key for generator's number}
	Let $G$ be a  $2$-group, and $H, K \leq G$  such that $[K,H]\leq K^{2}$. Then
	$$d(HK)\geq d(H)+d(K)-d(H\cap K);$$
	and the equality holds if and only if $(H\cap K)^{2}=H^2\cap K^2$.
\end{proposition}
Indeed,   the assumption $[K,H]\leq K^{2}$ implies that $HK\leq G$ and $(HK)^{2}=H^{2}K^{2}$. Hence,
$$\vert HK: (HK)^{2}\vert =\frac{\vert H \vert \cdot  \vert K \vert}{\vert H \cap K \vert} \frac{\vert H^{2}\cap K^{2} \vert }{\vert H^{2} \vert \cdot\vert K^{2} \vert }, $$
therefore,
$$2^{d(HK)}=2^{d(H)+d(K)}\frac{\vert H^{2}\cap K^{2} \vert}{\vert H\cap K \vert}.$$
As $(H\cap K)^{2}\leq H^{2}\cap K^{2} $, we have $\frac{\vert H^{2}\cap K^{2} \vert}{\vert H\cap K \vert}\geq 2^{-d(H\cap K)}  $; the proposition is immediate now. 
\begin{proof}[Proof of Corollary \ref{Equivalence for d-maximality}]
The implication $(i)\Rightarrow (ii)$ is immediate by Lemma \ref{Basic facts}(2).   Assume $(ii)$, and let $H<G$ be an $A$-subgroup containing $\gamma_2(G)$.  Obviously, $\gamma_2(G)^{2}\leq H^{2}$,  that is $\gamma_3(G)\leq H^{2}$ (by Theorem \ref{Kahn's modified 1}).  Therefore $$d(H)=d(H/\gamma_3(G))<d(G/\gamma_3(G))=d(G),$$
so $(ii)\Rightarrow (iii)$.  Finally, assume $(iii)$, and let $H$ be a proper $A$-subgroup of $G$.  By Theorem \ref{Kahn's modified 1}, $[H,\gamma_2(G)]\leq  \gamma_2(G)^{2}$. Now, Proposition \ref{Key for generator's number} yields
$$d(H\gamma_2(G))\geq d(H)+d(\gamma_2(G))-d(H\cap \gamma_2(G)).$$
As $\gamma_2(G)$ is powerful (by Corollary \ref{Cor main 2}), we have  $d(\gamma_2(G))\geq d(H\cap \gamma_2(G))$ (see \cite[Theorem 2.9]{DDMS}); so $d(H\gamma_2(G))\geq d(H)$; and plainly $H\gamma_2(G)$ is  a proper $A$-subgroup of $G$, so $d(H\gamma_2(G))<d(G)$; thus $(i)$ holds.  
\end{proof}

\begin{proof}[Proof of Theorem \ref{Nonmetacyclic}]
(i) Set $N=G^4$. Obviously, $N\leq P_3(G)$; and by Corollary \ref{Cor main 2}, $P_3(G/N)=(G/N)^{4}=1$,  hence $P_3(G)\leq N$; this shows that $ P_3(G)= N$.  Similarly, we have  $\gamma_2(G/N)=P_2(G)/N$, so $\gamma_2(G) N=P_2(G)$; thus
$$P_2(G)=\gamma_2(G) P_2(G)^2[P_2(G),G].$$
It follows that $P_2(G)=\gamma_2(G)$ (see the introduction of Section \ref{By-product of being powerful}), so $G/\gamma_2(G)$ is elementary abelian, and subsequently the same is true for $\gamma_2(G)/\gamma_3(G)$ (see the proof of Lemma \ref{Basic facts}(4)).  Thus,  $P_3(G)= \gamma_3(G)$, and so
$G/\gamma_3(G)$ is $d$-maximal for $A$-subgroups;  the result  is immediate now by Corollary \ref{Equivalence for d-maximality}.
	
(ii) The assumption implies that $G/G^{4}$ is a $d$-maximal, and $d(G)=d(G/G^{4})=3$. By (i), $G$ is likewise $d$-maximal, so  by Proposition \ref{dMax 2Grp on 3-gen have class 2}, $G$ has class $\leq 2$, that is $G^{4}=1$; the result follows. 

\end{proof}

\section{Further results}\label{Futher Results}
We start by two general facts that we need in the sequel.
\begin{lemma}\label{Normality of secondMax subgroups} 	If $G$ is a $p$-group in which every maximal subgroup $M$ satisfies $d(M) = d(G) - 1$,  then all the subgroups of index  $p^{2}$ in $G$ are normal.
\end{lemma}
Indeed, let $H \leq G$ be of index $p^{2}$, and $M$ be a maximal subgroup containing $H$.  Then $H$ is maximal in $M$,  so $\Phi(M) \leq H$.  On the other hand, we have 
$\vert G : \Phi(G)\vert=p\vert M : \Phi(M)\vert $,  so $\Phi(M) = \Phi(G)$.  Thus
$\gamma_{2}(G) \leq H$; the result follows.

\begin{lemma}\label{Omega in d-max p odd}
	If $G$ is a non-abelian $d$-maximal $p$-group, with $p>2$, then
	\begin{itemize}
		\item[(i)] $\Omega_{1}(G)$ is not abelian.
		\item[(ii)] $d(\gamma_{2}(G)) \leq d(G) - 2$.
	\end{itemize}
\end{lemma}
If  $\Omega_{1}(G)$ is abelian, then clearly $\vert \Omega_{1}(G)\vert = p^{d(\Omega_{1}(G))}$.  Also,  $G$ has class $2$  and $\gamma_2(G)$ is elementary abelian (by Theorem \ref{Main 1}, or Lemma \ref{Basic facts}(4)), so the map $x\mapsto x^p$ is an endomorphism of $G$.  It follows that
$$p^{d(G)} = \left| G : \Phi (G) \right| \leq \left| G : G^{p} \right|  = \left| \Omega_{1}(G )\right| = p^{d(\Omega_{1}(G))}.$$ 
Hence, $d(G) \leq d(\Omega_{1}(G))$, a contradiction; this proves (i).	Next,  we have $\gamma_{2}(G) < \Omega_{1}(G)$ by  (i), so one can find
$ x \in G\setminus \gamma_{2}(G)$ of order $p$.   It follows that
$H = \gamma_{2}(G)\left\langle x \right\rangle $ is elementary abelian of rank $d(\gamma_2(G))+1$.  If (ii) is false, then we would have  $G =H$, and so $G = \left\langle x \right\rangle $, a contradiction.  
\medskip

\begin{proposition}\label{dMax p-groups, 3 gen and p odd}
If $G$ is a $d$-maximal non-abelian $p$-group, with $d(G)=3$ and $p>2$, then $G$ is isomorphic to the central product  $C_{p^2}*S$, where $S$ is the extraspecial $p$-group of order $p^3$ and exponent $p$; in other words 
	$$G = \langle a, b, c \mid a^{p^{2}}=b^{p}=c^{p}=1, [ a, b]=[a, c]=1,  [ b, c] = a^{p} \rangle . $$
\end{proposition}

\begin{proof}
By Theorem \ref{Main 1} and Lemma \ref{Omega in d-max p odd}(ii), we have $\vert \gamma_2(G)\vert =p$,  so $\vert G\vert =p^{4}$.  Moreover, $\Omega_{1}(G)$ has order $p^3$ as  $\vert G : \Omega_{1}(G) \vert = \vert G^{p} \vert =p$, and  $\Omega_{1}(G)$ is not abelian by Lemma \ref{Omega in d-max p odd}(i).  Let $A\lhd G$ be   elementary abelian of rank $2$; as $A\leq \Omega_1(G)$, we have $A \nleq Z(G)$, hence $M= C_{G}(A)<G$.  Also, $G/M$ embeds naturally in $\mathrm{Aut}(A)\cong \mathrm{GL}_{2}(\mathbb{F}_p)$,   so $\vert G : M \vert = p$, that is $\vert M\vert =p^3$;  and $M$ is abelian because $A\leq Z(M)$.  Choose $v\in A$ and $w\in \Omega_1(G)$ such that $\Omega_1(G)=\langle v,w\rangle$. As $[v,w]\neq 1$, one can find $u\in M$ such that $[v,w]=u^p$.   Now,  $u$ has order $p^2$, so $\langle u\rangle \lhd G$ by Lemma \ref{Normality of secondMax subgroups}; hence  $[u,w]=u^{ip}$ for some $0\leq i<p$. If $i=0$, set $a=u$, $b=v$ and $c=w$; and if $i\neq 0$, set $a=uv^{-i}$, $b=v^i$ and $c=w^j$, where $ij=1 \mod p$.   It follows that $G=\langle a,b,c\rangle$, and
$a^{p^2}=b^p=c^p=1$, $[a,b]=[a,c]=1$, and $[b,c]=a^p$, as desired.  Note that $\Omega_1(G)= S$, and $G=\langle a\rangle * S$. (It should be clear that the latter group is $d$-maximal.)

\end{proof}
The next result is due to Y. Berkovich (see \cite[Proposition 10.17]{Berk1}); its proof is direct, and follows easily by induction on $\vert G\vert $. 
\begin{lemma}\label{Berkovich on self-centralizing groups}
	Let $G$ be a $p$-group.  If  $ G$ contains a non-abelian self-centralizing subgroup of order  $p^3$, then $G$ is of maximal class; in particular $d(G)=2$.
\end{lemma}
Note that this lemma can be used to prove Proposition \ref{dMax p-groups, 3 gen and p odd} (indeed, pick $a\in C_G(\Omega_1(G))\setminus \Omega_1(G)$, and observe that $G=\Omega_1(G)*\langle a\rangle$). It is also  useful in proving the next result, and completing the proof of Theorem \ref{Class of 3-gen d-max 2-grp}.

\begin{proposition}\label{dMax 2Grp on 3-gen have class 2}
If $G$ is a $d$-maximal $2$-group with $d(G)=3$, then $G$ has nilpotency class $\leq 2$.

\end{proposition}
\begin{proof}

	Assume that $G$ is a  counterexample of minimal order, so  clearly  $\vert \gamma_3(G)\vert=2$.  As   $\vert  G/\gamma_2(G)\vert =2^{3}$, and $\vert 
	\gamma_2(G)/\gamma_3(G)\vert\leq 2^{2}$ (see Lemma \ref{Basic facts}(4)), we have $\vert G\vert=2^5 \mbox{ or } 2^6$.  
	
If $H\leq G$ is non-abelian of order $8$, then $C_G(H) \nsubseteq H$ by Lemma \ref{Berkovich on self-centralizing groups}; choose therefore   $x\in C_G(H)\setminus H$ of minimal order,  and let $K=H\langle x \rangle$. Plainly, $K^{2}=Z(H)$ has order $2$,  so $d(K)\geq 3$, but $\vert K\vert \leq 2^4$, a contradiction.  Thus, every subgroup of order $8$ in $G$ is abelian.

	Assume $\vert G\vert= 2^{5}$, so $\vert \gamma_2(G)\vert =2^2$.  If we choose  $x\in G$ which does not commute with $\gamma_2(G)$, then $\gamma_2(G)\langle x \rangle$ is  non-abelian  of order order $8$, a contradiction. So $\vert G\vert =2^{6}$, and   $\vert \gamma_2(G)\vert =2^3$. 
	
Let $A=\Omega_1(\gamma_2(G))$ and  $C=C_G(A)$.  As $\gamma_2(G)$ is abelian and  $d(\gamma_2(G))=2$ (by Theorem \ref{Kahn's modified 1}),  we have $A$ is elementary abelian of rank $2$.  Since   $G/C$ embeds naturally in $\mathrm{Aut}(A)$, we have $\vert G:C\vert\leq 2$. Assume $C=G$ (that is $A\leq Z(G)$), and  let  $N\leq A$ be of order $2$ with $N\neq \gamma_3(G)$. Then $N\lhd G$, and   $G/N$ is $d$-maximal of class $3$ (by Lemma \ref{Basic facts}(2)); this contradicts the minimality of $G$, so   $\vert G:C\vert =2$. Moreover, if $G$  contains an involution $t$ outside $A$, then   $A\langle t\rangle$ has order $8$, so  it is  $A\langle t\rangle$  is elementary abelian of rank $3$, a contradiction. The latter shows that $\Omega_1(G)=A$.

If $x\in G\setminus C$, then $x$  has order $8$,  and  $C_G(x)=\langle x \rangle$; moreover, we have $Z(G)\leq  \langle x^2 \rangle$.  Indeed,  $A\langle x\rangle$  is obviously non-abelian, so  its order is $\geq 2^4$; and  since  $x^4\in A$, we have $\vert A\langle x\rangle\vert =2^{4}$. If $x^4=1$, then  $x^{2}$ would lie in $A$, which implies that 	$A\langle x\rangle$ has  order $8$, a contradiction.  This proves that $x$ has order $8$.  Also, we have $\vert C_G(x)\vert \leq 8$, as otherwise  $C_G(x)$ would contain $\gamma_2(G)$ (by Lemma \ref{Normality of secondMax subgroups}), so $x$ commutes in particular with $A$, a contradiction.  It follows that $C_G(x)=\langle x \rangle$ (since $x$ has order $8$), and $Z(G)\leq  \langle x^2 \rangle$ (as $x\notin Z(G)$).  

Choose $x_1\in G\setminus C$, and let $B_1 =  \langle x_1^2 \rangle$.  Obviously, $A$ and $B_1$ are maximal subgroups of $\gamma_2(G)$, so it remains a third which we denote by $B_2$.   Note that $B_1$ and $B_2$  contain $\gamma_3(G)$, so they are normal in $G$, and both of them are cyclic of order $4$. Let $C_i=C_G(B_i)$;  as $G/C_i$ embeds in $\mathrm{Aut}(B_i)$, we have $\vert G:C_i\vert \leq 2$.

Assume  that $\vert G:C_1\vert = 2$. Let $L=C\cap C_1$; so $L=C_G(\gamma_{2}(G))$ because $\gamma_2(G)=AB_1$.  Observe also that $L=C\cap C_1$ and $L=C_1\cap C_2$ (because $AB_2=B_1B_2=\gamma_2(G)$).  Next, we have  $\vert L\vert =2^4$ (since $C_1\neq C$),  and  $d(L)=2$ since $L$ contains $\gamma_2(G)$.  It follows that $\vert L^2\vert=4$, that is  $L^2$ is one of $B_1$, $B_2$ or $A$.    Suppose $L^2=B_1$.  If $x^2\in B_1$ for every $x\in C_1\setminus L$, then $C_1^2=B_1$, a contradiction.  Therefore, there exists $x\in C_1\setminus L$ such that $x^2\notin B_1$.  But  $x$ has order $8$ (since $x\notin C$),  so $\langle x^2\rangle =B_2$, and subsequently $x\in  C_1\cap C_2=L$, a contradiction.  Similarly, if we assume 
$L^2=B_2$, then we can find $x\in C_2\setminus L$ such that $\langle x^2\rangle =B_1$, which implies that $x\in C_1\cap C_2=L$, a contradiction.  So we have to suppose $L^2=A$.  We know that $C$ contains an element $x$ of order $8$ (as $C^2=\gamma_2(G)$), so $\langle x^2\rangle$ is equal to $B_1$ or $B_2$, hence $x\in C\cap C_1$ or $x\in C\cap C_2$, that is $x\in L$, a contradiction.  

The last paragraph shows that $C_1=G$, that is $B_1=Z(G)$.  It follows that $C_G(\gamma_2(G))=C$, and so $[C,C]\leq B_1$ (by the three subgroups lemma). Also, we have  $y^2\in B_1$ for all $y\in G\setminus C$.  Indeed,  otherwise we would have $\gamma_2(G)=\langle y^2\rangle  B_1$ for some $y\in G\setminus C$,  so $C_G(y^2)=C$, which contradicts the fact that $y\notin C$.  Let $\overline{G}=G/B_1$; so $\overline{G}$ is $d$-maximal by Lemma \ref{Basic facts}(2), and $d(\overline{G})=3$.  Let $M\neq C$ be a maximal subgroup of $G$ such that $\overline{M}=G/B_1$ is non-abelian (such an $M$ exists as otherwise all the proper subgroups of $\overline{G}$ are abelian, so if we choose two non-commuting elements $\overline {x},\overline{y}\in \overline{G}$, then $ \langle \overline {x},\overline{y}\rangle = \overline{G}$, a contradiction).  Now, by Lemma \ref{Berkovich on self-centralizing groups}, $\overline{G}=\overline{M}\langle \overline{z}\rangle $ for some $\overline{z}\in C_{\overline{G}}(\overline{M})$.  For every $y\in M\setminus C$, we have $\overline{y}^2=1$, so $\overline{M}\ncong Q_8$; thus $\overline{M}\cong D_8$.   Let $\overline{a}\in \overline{M}$ be of order  $4$, and  $\overline{b}\in \overline{M}$ be of order  $2$ such that $\overline{M}=\langle \overline{a},\overline{b}\rangle$.  If $\overline{z}^2=1$, then  $\langle \overline{a}^2,\overline{b}, \overline{z} \rangle$ is  elementary abelian of rank $3$, a contradiction.  Thus  $\overline{z}$ has order $4$.  Now,  $\overline{N}=\langle \overline{a},\overline{bz} \rangle$ is isomorphic to $Q_8$, but for $\overline{y}\in \overline{N}\setminus \overline{C}$, we have $\overline{y}^2=1$, which yields the final contradiction.

\end{proof}

Now, we shall prove Theorem \ref{Class of 3-gen d-max 2-grp}.
Plainly, the case (a) occurs exactly when $G$ is abelian  (i.e., when $\vert G\vert=2^3$). 

Assume   $\vert G\vert=2^4$, so $\vert\gamma_2(G) \vert =2$. Choose two non-commuting elements $x,y\in G$, and set $H=\langle x,y\rangle$. Then  $H$ is non-abelian, and $H<G$ as $d(H)=2$, so $\vert H\vert =2^3$.  By Lemma \ref{Berkovich on self-centralizing groups},  $C_G(H)\nleq H$; choose an element $z\in C_G(H)\setminus H$.  If $z^2=1$, then $G=H\times \langle z\rangle$,   and $H\cong Q_8$  (otherwise, $H$ would contain a subgroup $A$ of the form $C_2\times C_2$, so $d(A\langle z\rangle )=3$, a contradiction); hence $G$ is as in (b).  If $z$ has order $4$, then  $G$ is clearly isomorphic to  (c).

Assume $\vert G\vert= 2^{5}$.  As $G$ has class $2$,  $\gamma_2(G)$  is elementary abelian of rank $2$ (by Theorem \ref{Kahn's modified 1}), and subsequently  $\Omega_1(G)=\gamma_2(G)$.  Let $x_1,x_2,x_3$ be  generators of $G$.  Then $[x_i,x_j]$, with $i<j$, generate  $\gamma_2(G)$; but $\dim_{\mathbb{F}_2} \gamma_2(G)=2$,  so there are $i, j, k\in \mathbb{F}_2$, not all zeros, such that
$$[x_1,x_2]^{i} [x_1,x_3]^{j} [x_2,x_3]^{k}=1.$$
If $i=0$, set  $u=x_1^jx_2^k$ and $v=x_3$; and if $i=1$, set $u=x_1x_3^{-k}$ and $v=x_2x_3^j$.  So $[u,v]=1$, that is $M=\langle u,v\rangle$ is an abelian.  We have $u,v\notin \gamma_{2}(G)$, so both of them have order $4$, and $u^2\neq v^2$, as otherwise $(uv)^2=u^2v^2=1$, a contradiction.  It follows that $M\cong C_4\times C_4$.

Choose $c\in G\setminus M$. The map $x\mapsto [x,c]$, from $M$ to $\gamma_2(G)$,  is clearly surjective, and its kernel coincides with $Z(G)$; hence $\vert Z(G)\vert =4$,  so $Z(G)=\gamma_2(G)$.  It follows that $C_G(x)=M$, and $[c,x]\neq x^2$ for every 
for  $x\in M\setminus \gamma_2(G)$ (otherwise, we would have $(xc)^2=c^2$, so  $\langle c,xc\rangle$ has order $8$; but the latter  contains $\gamma_2(G)$ by Lemma \ref{Normality of secondMax subgroups}, so $\langle c,xc\rangle$ is abelian, that is $c\in C_G(x)$, a contradiction).

Now, choose $a,b\in M$ such that $c^2=a^2$, and  $b^2=[a,c]$.  As we have seen above, $a^2\neq b^2$, so $M=\langle a,b\rangle$, that is  $G=  \langle a ,b,c\rangle$.  Also, $[b,c]\neq b^2$, so $[b,c]=a^2$ or $[b,c]=a^2b^2$.  If the latter holds, then $(cb)^2=c^2b^2a^2b^2=1$, a contradiction.  In conclusion, we have $a^4=b^4=[a,b]=1$, $a^2=c^2$, $[a,c]=ab^2$ and $[b,c]=ba^2$, as desired (that is $G$ is isomorphic to (d)).

\section{Concluding remarks and problems}\label{Last section}
Using $\mathtt{GAP}$ (\cite{GAP}), the  following code  checks whether a given $p$-group $G$ is $d$-maximal.  
\begin{verbatim}
IsDNMax:= function (G)
local nor, max;
nor := NormalSubgroups(G);
nor := Filtered(nor, x -> Size(x)<Size(G) and IsSubgroup(x, DerivedSubgroup(G)));
max := Maximum(List(nor, x -> RankPGroup(x)));
return max < RankPGroup(G);
end;
\end{verbatim}
(Note that we are using Corollary \ref{Equivalence for d-maximality}  to reduce the calculation.) 
\medskip

When working with a $\mathtt{list}$ of $p$-groups, the following  returns the sub-list formed by the ones which are $d$-maximal.
\begin{verbatim}
IsDNMaxList:= function(list)
local Li, i;
Li:=[];
for i in [1..Length(list)]  do 
if IsDNMax(list[i])
then Add(Li, list[i]);
fi;
od;
return Li;
end;
\end{verbatim}

Now, with the Small Groups library of $\mathtt{GAP}$, we can determine all the $d$-maximal  $2$-groups of order $\leq 2^8$.  Among these, there are exactly two which have (nilpotency) class $3$, namely 	$G_{1} =\mathtt{SmallGroup}(2^8, 22218)$, and $G_{2} = \mathtt{SmallGroup}(2^8, 22219) $ (both of them have order $2^8$).  We can find many other examples of higher order by taking, for instance, quotients of the $2$-covering group of each of $G_1$ and $G_2$ (it turned out that all of them have class $3$).  The remaining $d$-maximal $2$-groups have class $2$ (except of course the elementary abelian $2$-groups of rank $\leq 8$); there are too many of them as the following table shows (below, $G$ is said to be of type  $(a,b)$, if $d(G)=a$ and $d(G^{2})=b$):

\begin{center}
	\begin{tabular}{ l | c  }
		\hline
		Type & Number of $d$-maximal $2$-groups \\ \hline
		$(2,1)$ & $1$ \\ \hline
		$(3,1)$ & $2$ \\ \hline
		$(3,2)$ & $1$ \\ \hline
		$(4,1)$ & $4$ \\ \hline
		$(4,2)$ & $19$\\ \hline
		$(4,3)$ & $25$ \\ \hline
		$(5,1)$ & $5$ \\ \hline
		$(5,2)$ & $98$ \\ \hline
		$(5,3)$ & $6808$ \\ \hline
		$(6,1)$ & $7$ \\ \hline
		$(6,2)$ & $362$ \\ \hline
		$(7,1)$ & $8$ \\ \hline
	\end{tabular}
\end{center}
\medskip
Note that in the   cases where $d(G)\leq 4$, the above results have been established  much earlier by Kahn and Patentreger, but unfortunately their manuscript \cite{Kahn Unpublsh} remains unpublished.\footnote{According to Kahn, a small mistake in their calculation  has been corrected by P. M. Minh and P. D. Tai.}

\medskip
As far as we know, all the known $d$-maximal $2$-groups are of class $\leq 3$. 

\begin{pbm}
	Are there $d$-maximal $2$-groups of arbitrary large nilpotency class?
\end{pbm}

It seems that we always  need a larger number of generators  to produce $d$-maximal $2$-groups of higher classes.
\begin{conjecture}
For every integer $r\geq 1$, the (nilpotency) class of an $r$-generated $d$-maximal $2$-group  is bounded in terms of $r$.
\end{conjecture}
By virtue of Lemma \ref{Basic facts}(4), the above amounts to saying that the
$r$-generated $d$-maximal $2$-groups have an $r$-bounded order.
\medskip

Call a finitely generated pro-$p$ group $G$ \textit{$d$-maximal}, if   $d(H)<d(G)$ for every  closed subgroup $H<G$ (recall that $d(H)$  stands for $\dim_{\mathbb{F}_p}H/H^{p}[H,H]$).  The following is a weaker version of the previous conjecture.
\begin{conjecture}
Every finitely generated $d$-maximal pro-$2$ group is finite.
\end{conjecture}

Let us mention finally a problem posed by A. Mann \cite[Question 22]{MannQuestions}.
\begin{pbm}
What is the structure of $d$-maximal $2$-groups? Do they have bounded derived length?
\end{pbm}

% BIBLIOGRAPHY

\section*{Acknowledgements}
We are grateful to Professor Bruno Kahn for several invaluable comments, and for providing us with some of his earlier manuscripts on the subject.   The first author would like  to thank Professor B. Eick for her hospitality during his visit to Universit\"{a}t Braunschweig, and for the invaluable things he learned from her.


\begin{thebibliography}{00}
	%1
		\bibitem{Berk1} Y. Berkovich and Z. Janko, Groups of prime power order, vol.
	1, Walter de Gruyter, 2008.
	\bibitem{Berk3} Y. Berkovich and Z. Janko, Groups of prime power order, vol.
	3, Walter de Gruyter, 2011.
	\bibitem{Black} 
	N. Blackburn, Generalizations of certain elementary theorems on $p$-groups, \emph{Proc. London Math. Soc.} \textbf{11}  (1961),  1--22.
	
	\bibitem{DDMS} J. Dixon, M. du Sautoy, A. Mann, D. Segal, Analytic pro-$p$ Groups, second ed., Cambridge Univ. Press, 1999.
	\bibitem{Janko} 
	Z. Janko, Finite $2$-groups with exactly four cyclic subgroups of order $2$, \emph{J. reine angew. Math.} \textbf{566} (2004), 135--181.	
	\bibitem{Kahn SW}	
	B. Kahn, The total Stiefel-Whitney class of a regular representation,  \emph{J. Algebra} \textbf{144} (1991), 214--247.	
	
	\bibitem{Kahn APE}	
	B. Kahn, A characterization of powerfully embedded normal subgroups of a $p$-group, \emph{J. Algebra} \textbf{188}
	(1997), 401--408.
	\bibitem{Kahn Unpublsh}	
	B. Kahn, Sur les $p$-groupes d-maximaux, unpublished manuscript (1992) (communicated by Bruno Kahn).
	\bibitem{GAP}
	The GAP Group, GAP - Groups, Algorithms, and Programming, Version 4.4.12, www.gap-system.org, 2008.
	
	%\bibitem{GonJaik} 
	%J. Gonz\'alez-S\'anchez and B. Klopsch, On $w$-maximal groups, \emph{J. Algebra} \textbf{328} (2011), 155--166.	
	\bibitem{Laffey}
	T. J. Laffey, The minimal number of generators of a finite $p$-group, \emph{Bull. London Math.
		Soc. } \textbf{5} (1973), 288--290
	\bibitem{MannLubotzky}
	A. Lubotzky and A. Mann,  Powerful p-groups. I. finite groups,
	\emph{J. Algebra} \textbf{105} (1987) 484--505.
	%\bibitem{MannPowerStructure}
	%A. Mann, The power structure of p-groups, II., \emph{J. Algebra}  \textbf{318} (2007), 953--956.
	\bibitem{MannQuestions} A. Mann, Some questions about $p$-groups, \emph{J. Aust.
		Math. Soc.}, \textbf{67} (3) (1999), 356--379.
	
	\bibitem{Minh}
	P.A. Minh,
	$d$-Maximal $p$-groups and Stiefel-Whitney classes of a regular representation,
	\emph{J. Algebra} \textbf{179},  (1996),  483--500
	\bibitem{Minh1997}
	P.A. Minh, On a conjecture of Kahn for
	the Stiefel-Whitney classes of the regular representation,
	\emph{J. Algebra} \textbf{188},  (1997),  590--609
%	\bibitem{Thompson}
%	J. G. Thompson, A replacement theorem for $p$-groups and a conjecture,  \emph{J. Algebra} \textbf{13}  (1969), 149--151.
\end{thebibliography}
\end{document}